%% file: article.tex
\documentclass[review,hidelinks,onefignum,onetabnum]{siamart220329}
\nolinenumbers

\input{shared}

\usepackage[algo2e]{algorithm2e}

\ifpdf
\hypersetup{
  pdftitle={Shape optimization for the Laplacian eigenvalue over triangles and its application to interpolation error constant estimation},
  pdfauthor={Ryoki ENDO, Xuefeng LIU}
}
\fi

\newcommand{\CR}{{\mbox{\tiny CR}}}
\newcommand{\CG}{{\mbox{\tiny CG}}}
\usepackage{color}

\newcommand{\LIU}[1]{{\color{red}{#1}}}

\begin{document}

\maketitle

\begin{abstract}
A computer-assisted proof is proposed for the Laplacian eigenvalue minimization problems over triangular domains under diameter constraints.
The proof utilizes recently developed guaranteed computation methods for both eigenvalues and eigenfunctions of differential operators. 
The paper also provides an elementary and concise proof of the Hadamard shape derivative, which helps to validate the monotonicity of eigenvalue with respect to shape parameters. Beside the model homogeneous Dirichlet eigenvalue problem, the eigenvalue problem associated with a non-homogeneous Neumann boundary condition, which is related to the Crouzeix--Raviart interpolation error constant, is considered. The computer-assisted proof tells that among the triangles with the unit diameter, the equilateral triangle minimizes the first eigenvalue for each concerned eigenvalue problem.
\end{abstract}



\section{Introduction}

Eigenvalue problems of the Laplacian over triangular domains have been well investigated in the long history of mathematics. For example, in the analytical analysis of the Laplacian spectrum, triangular domains have been studied to conjecture behaviors of eigenvalues over general domains; see, e.g., the early work of P{\'o}lya--Szeg\"{o}    \cite{polya1951isoperimetric} and recent results \cite{arbon2022global,henrot2006extremum,henrot2017shape, 534dfed29e6f41038665a25c2e995659, siudeja2010isoperimetric}. In 1951, P{\'o}lya--Szeg\"{o}  show that the first Dirichlet eigenvalue takes the minimum value at a regular $n$-polygon among all the $n$-polygons ($n=3,4$) with a fixed area \cite{polya1951isoperimetric}. As a consequence of this result and the isoperimetric inequality, it is proved that the equilateral $n$-polygon($n=3,4$) minimizes the first Dirichlet eigenvalue among the triangles with a given diameter \cite{arbon2022global,534dfed29e6f41038665a25c2e995659,polya1951isoperimetric,siudeja2010isoperimetric}.
In the field of numerical analysis, the eigenvalue problems of differential operators is utilized to obtain concrete bounds of various interpolation error constants that are desired in the error analysis for finite element methods; see, e.g., \cite{Kikuchi+Liu2007,liu-kikuchi-2010}. 

In this paper, the shape optimization for two Laplacian eigenvalue problems with Dirichlet and Neumann boundary conditions are studied. It is proved that, under the diameter constraint, the regular triangle gives the minimum value of the first Laplacian Eigenvalue among triangular domains. The case of the Dirichlet eigenvalues is a reconfirmation of the consequences of \cite{polya1951isoperimetric}, while the case of the non-homogeneous Neumann eigenvalues is a new result that validates the numerical observation in studying the Crouzeix--Raviart interpolation error constant \cite{liu2015framework}.

The method proposed in this paper has the following novelties.
\begin{itemize}
    \item The technique used in the elementary proof for the Hadamard shape derivative can be used to handle special setting of domain shapes and boundary value conditions. For example, the eigenvalue problem related to the Crouzeix--Raviart interpolation error constant has a non-homogeous boundary value condition.
    \item The utilization of guaranteed estimation of eigenvalues and eigenfunctions of differential operators make it possible to solve the problems that is difficult to process by pure theoretical analysis. Although only triangular domains are discussed in this paper, the  techniques developed here have the possibility to provide the proof for the conjecture of P{\'o}lya--Szeg\"{o} about the shape optimization of $n$-polygon domains ($n\ge 5$). 
\end{itemize}


\medskip

The problem of minimizing the first non-homogeneous Neumann eigenvalue originates from the error analysis for the Crouzeix--Raviart finite element method; see \cite{liu2015framework}.
Given a triangle $T$, 
let $\Pi_h:H^1(T)\to L^2(T)$ be the Crouzeix--Raviart interpolation such that, for $u\in H^1(T)$, 
$\Pi_h u$ is a linear function satisfying
\begin{equation*}
\int_{e_i}\Pi_h u-u~ds=0,~~i=1,2,3,
\end{equation*}
where $e_i$'s stand for the edges of $T$.
The following  error estimation of $\Pi_h u$ is available  with a quantity $C(T)$ that only depends on the shape of $T$:
\begin{equation*}
\|u-\Pi_h u\|_{T} \le C(T) 
\|\nabla(u-\Pi_h u)\|_{T} ~~ \forall u \in H^1(T).
\end{equation*}
Here, the optimal $C(T)$ for a specific $T$ is defined by
\begin{equation}
\label{def:C_T}
C(T):=\sup_{u\in H^1(T)} \frac{\|u-\Pi_h u\|_{T}}{\|\nabla(u-\Pi_h u)\|_{T}}~.    
\end{equation}
Note that the value of $C(T)$ is linearly dependent on the diameter of $T$ and its value is determined by the first eigenvalue of the Laplacian with a certain boundary condition as stated in \eqref{eq:eig-with-neumann-bdc}.

In  \cite{liu2015framework}, it is shown through rigorous numerical computation that 
$C(T)$ has a uniform bound 
$C(T) \le 0.1893 h_T$ for triangles of arbitrary shapes, where $h_T$ denotes the diameter of $T$.
Such a bound plays an important role in applying the Crouzeix--Raviart FEM to obtain the following explicit lower eigenvalue bounds for the Laplacian eigenvalue problems: 
$$
\lambda_{k} \ge \frac{\lambda_{k,h}}{1+(0.1893 h)^2\lambda_{k,h}},\quad 
k=1, \cdots, \mbox{dim}(V_h^{\CR}) ~.
$$
Here, $\lambda_{k}$'s are the Laplacian eigenvalues defined over the objective domain $\Omega$; $V_h^{\CR}$ is the Crouzeix--Raviart FEM space over a triangulation mesh of $\Omega$; 
$\lambda_{k,h}$'s are the approximate eigenvalues solved in $V_h^{\CR}$; $h$ is the mesh size of the triangulation of the domain.

It is of great interest when the equality in the estimation ``$C(T)\le 0.1893 h_T$" holds. Numerical computation implies that, for a triangle element with a fixed diameter, the ``worst case", i.e., the maximum value of $C(T)$ happens when the triangle is equilateral.
However, this statement is not strictly proved yet.
In this paper, by solving the shape optimization of the Laplacian eigenvalue over triangular domains, we will prove that the equilateral triangle gives the maximum value of $C(T)$ among triangles of arbitrary shapes.  
The code for the computer-assisted proof is available at \url{https://ganjin.online/ryoki/ShapeDerivativeEstimation}.

\medskip

{The remainder of this paper is organized as follows: In Section \ref{section:preliminary}, we introduce the objective eigenvalue problems  along with the theorems that provide lower eigenvalue bounds and estimation for eigenvector approximation. In Section \ref{section:main-theories}, we describe the outline of the computer-assisted proof for the shape optimization problem, and the error estimation for the shape derivative of the eigenvalue. In Section \ref{section:numerical-results}, the computation results are presented. Finally, in Section \ref{section:conclusions}, we state our conclusions.}

\section{Preliminary}\label{section:preliminary}

Let us introduce the notation to be used in this paper.
For a triangular domain $T\subset\mathbb{R}^2$, let $V(T)$ be either of the following function spaces:
\begin{gather*}
    V_0(T):=\{v\in H^1(T): v=0 \mbox{ on } \partial T\},\\
    V_e(T):=\{v\in H^1(T): \int_{e_i}v \ \mbox{ds}=0 \mbox{ for }i=1,2,3 \}
\end{gather*}
where $e_1,e_2,e_3$ stand for the distinct edges of $T$, respectively. Let $(\cdot,\cdot)_T$ be the $L^2$--inner product for $L^2(\Omega)$ and $\{L^2(\Omega)\}^2$. 
Note that due to the boundary conditions of the spaces $V_0(T)$ and $V_e(T)$, $(\nabla\cdot,\nabla\cdot)_{T}$ 
is an inner product for each space. The argument in this paper will use the following norms for functions in $V$:
$$
\|u\|_{V(T)}:=(\nabla u,\nabla u)_T^{1/2},\quad
\|u\|_{T}:=(u, u)_T^{1/2}~.
$$
For simplicity, $\|u\|_{T} $ is also abbreviated as  $\|u\|$. 
In this paper, the following eigenvalue problem is considered:

Find $u\in V(T)$ and $\lambda>0$ such that
\begin{equation}
\label{eq:eigenvalue-problem}
(\nabla u,\nabla v)_{T}=\lambda (u,v)_{T} \hspace{1em} \forall v\in V(T).
\end{equation}
The above eigenvalue problem has the following eigenvalue distribution:
\begin{equation*}
0<\lambda_1(T)\leq\lambda_2(T)\leq\lambda_3(T)\leq\cdots
\end{equation*}
In the case of $V_e(T)$, the eigenvalue problem (\ref{eq:eigenvalue-problem}) in the operator formulation is as follows:
\begin{equation}
    \label{eq:eig-with-neumann-bdc}
-\Delta u = \lambda u \mbox{ in } T, \quad 
\left.\frac{\partial u}{\partial n}\right |_{e_i} = c_i ~~(i=1,2,3).
\end{equation}
Here, $c_i$'s are constants to be determined from the variational formulation (\ref{eq:eigenvalue-problem}).

Throughout this paper, the following shape optimization problem is considered:
Find a triangular domain $T^*$ such that
\begin{equation*}\label{shape-optimization-problem}
    T^*\in\mathcal{O},\hspace{1em} \lambda_1(T^*)=\min_{T\in\mathcal{O}}\lambda_1(T),
\end{equation*}
where $\mathcal{O}$ is a class of triangular domains in $\mathbb{R}^2$ with unit diameter.

\medskip

The finite element method (FEM) will be utilized to solve the eigenvalue problem over triangles. Let us introduce the FEM approximation to the eigenvalue problem (\ref{eq:eigenvalue-problem}). Let $h$ be the maximal edge length of $\mathcal{T}^h$. 
Let $V_h^{\CG}(\subset V(T))$ be the discretized space of $V(T)$ through linear conforming  Galerkin FEM. Let $V_h^{\CR}(T)$ the discretized space of $V(T)$ through Crouzeix--Raviart (CR) FEM. For $u_h\in V_h$, the discretized gradient operator, still denoted by $\nabla$, is required. The details of the finite element spaces are provided in (\ref{def:fem-space-cg}), (\ref{def:fem-sapce-cr}).

To estimate upper and lower bounds of the exact eigenvalues $\lambda_k(T)$, the following two discrete eigenvalue problems are considered:

\begin{itemize}
    \item[(a)] Find $u_h\in V^{\CG}_h(T)$ and $\lambda_h>0$  such that
        \begin{equation*}
        (\nabla u_h,\nabla v_h)_{T}=\lambda_h (u_h,v_h)_{T} \hspace{1em} \forall v_h\in V^{\CG}_h(T).
        \end{equation*}
    \item[(b)] Find $u_h\in V_h^{\CR}(T)$ and $\lambda_h>0$  such that
        \begin{equation*}
        (\nabla u_h,\nabla v_h)_{T}=\lambda_h (u_h,v_h)_{T} \hspace{1em} \forall v_h\in V_h^{\CR}(T).
        \end{equation*}
\end{itemize}
Let  $N_{1}=\mbox{dim}(V_h^{\CG}(T))$, $N_2 =\mbox{dim}(V_h^{\CR}(T))$.
The eigenvalues of (a) are denoted by
\begin{equation*}
    (0<)~\lambda_{1,h}^{\CG}(T)\leq\lambda_{2,h}^{\CG}(T) \leq \cdots \leq \lambda_{N_1,h}^{\CG}(T)~,
\end{equation*}
 and the eigenvalues of (b) are denoted by
\begin{equation*}
    (0<)~\lambda_{1,h}^{\CR}(T)\leq\lambda_{2,h}^{\CR}(T)
    \leq \cdots \leq\lambda_{N_2,h}^{\CR}(T)~.
\end{equation*}
By using FEM approximation, we can evaluate eigenvalues with rigorous upper and lower bounds.
\begin{lemma}
\label{lem:l-estimation}
Let $C_h=0.1893h$, where $h$ is the mesh size of $\mathcal{T}^h$.
We have 
\begin{equation*}
    \underline\lambda_k:=\frac{\lambda_{k,h}^{\CR}(T)}{1+C_h^2\lambda_{k,h}^{\CR}(T)}\leq\lambda_k(T)\leq\lambda_{k,h}^{\CG}(T)=:\overline\lambda_k        
    \hspace{1em} 
    ~~ (k=1,2, \cdots, N_2)~.
\end{equation*}    
\end{lemma}

\begin{proof}
The lower eigenvalue bounds are provided in \cite{liu2015framework}, while the upper eigenvalue bounds are from the min-max principle since $V_h^{\CG}\subset V_0$.
\end{proof}

The eigenfunction of the eigenvalue problem can also be well approximated by FEM solutions. Let us introduce distances to measure the approximation error of eigenfunctions. Given two subspaces $E$ and $\widehat{E}$ of $V(T)$, the distance 
$\overline{\delta}(E,\widehat E)$ and  $\delta_b(E,\widehat E)$  are defined by
\begin{gather*}
    \overline{\delta}(E,\widehat E):=\max_{\hat v\in \widehat E,\|\hat v\|=1}\min_{v\in E,\|v\|=1}\|\nabla v-\nabla\hat v\|,\\
    \delta_b(E,\widehat E):=\max_{v\in E,\|v\|=1}\min_{\hat v\in\widehat E}\|v-\hat v\|.
\end{gather*}
The lemma below summarizes the results of Theorem 1 and estimation (49) of \cite{liu2022fully}. Note that
the estimation proposed in \cite{liu2022fully} can also handle the case of clustered eigenvalues, which is an extension of the result of \cite{birkhoff1966rayleigh}. Other error estimations with different distances are available in \cite{cances2020guaranteed}.

\begin{lemma}[Estimation 
 for eigenvector approximation]\label{lem:original_eigenvec_estimation}
Let $u\in V(T)$ be an eigenfunction corresponding to $\lambda_1(T)$ with $\|u\|=1$. Let $u_h\in V_h^{\CG}(T)$ be an approximate eigenfunciton of $u$ such that $\|u_h\|=1$.
Let $E:=\mbox{span}\{u\},~~\widehat E:=\mbox{span}\{u_h\}$ and  $\lambda_{1,h}:=\|\nabla u_h\|^2$. 
Assume $\lambda_{1,h}<\rho\leq\lambda_2(T)$. Then we have
\begin{gather}
\label{eq:eig-vec-est}
    \overline{\delta}^2(E,\widehat E)\leq \lambda_1+\lambda_{1,h}-2\lambda_1\sqrt{1-\delta_b^2(E,\widehat E)},~~ \delta^2_b(E,\widehat E)\leq\frac{\lambda_{1,h}-\lambda_1}{\rho-\lambda_1}~.
\end{gather}

\end{lemma}

For $a,b,c\in\mathbb{R}$ with $a\leq b< c$, introduce quantity $\eta(a,b,c)$ such that
\begin{equation}\label{eq:def-eta}
    \eta(a,b,c):=\left\{a+b-2a\sqrt{\frac{c-b}{c-a}}\right\}^{1/2}~.
\end{equation}
Then the estimation (\ref{eq:eig-vec-est}) becomes 
\begin{equation}
\label{eq:eig-vec-est-simp}
\overline{\delta}(E,\widehat E) \leq \eta(\lambda_1,\lambda_{1,h},\rho).
\end{equation}

\begin{remark}

The result of Lemma \ref{lem:original_eigenvec_estimation} is an easy-to-implement method to estimate the eigenfunction approximation error.
For $V_h^\CG$ using piecewise linear polynomial, the optimal convergence rate of $\delta_b$ is $O(h^2)$, while the 
estimation of \eqref{eq:eig-vec-est} only provides a sub-optimal convergence rate as $O(h)$.  To have a sharper bound of $\delta_b$, one can use Algorithm II of \cite{liu2022fully} where the residue error of the approximate eigenfunction is utilized to 
provide optimal estimation with convergence rate as $O(h^2)$. 


\end{remark}

\section{Main theories} \label{section:main-theories}
Let $T$ be a triangular domain with vertices as
$O(0,0)$, $A(1,0)$ and $B(a,b)$.
Since the eigenvalues of the Laplacian are isometry invariant, without loss of generality,  $B$ is further assumed to be located in area $\Omega_0(\subset\mathbb{R}^2)$ defined by
\begin{equation*}
    \Omega_0:=\{(x,y)\in\mathbb{R}^2:x^2+y^2\leq 1, x\geq\frac{1}{2},y>0\}.
\end{equation*}
In case  $B=(\cos\theta,\sin\theta)$,  $\theta=\angle AOB$, let $T^\theta:=T$ and $\lambda^\theta_k:=\lambda_k(T^\theta)$ ($k=1,2,...$). 

Below, we describe the outline of the computer-assisted proof for the shape optimization problem.
Since the value of constant $\lambda_1(T)$  depends on two parameters: $a, b$, the proof for the shape optimization problem consists of the following three steps:

\begin{itemize}
    \item [Step 1] Reduce the parameters $a,b$ to $\theta$ by taking the use of the monotonicity of $\lambda_1$ with respect to $b$. That is, 
\end{itemize}
\begin{equation*}
    \min_{T\in\mathcal{O}}\lambda_1(T)=\min_{\theta\in(0,\frac{\pi}{3}]}\lambda_1(T^\theta)~.
\end{equation*}
\begin{itemize}
    \item [Step 2] For $\theta \in (0, \pi/3-\varepsilon]$ ($\varepsilon$: an explicit value to be determined), the range of $\lambda_1^\theta$ is rigorously estimated to validate that $\lambda_1^\theta>\lambda_1^\frac{\pi}{3}$ for all $\theta\in (0,\pi/3-\varepsilon]$; see Fig.\ref{fig:graph_lambda1}.
    \item [Step 3] For $\theta \in [\pi/3-\varepsilon, \pi/3]$, it is proved that $\lambda_1^\theta$ is monotonically decreasing with respect to parameter $\theta$, where the derivative of the eigenvalue w.r.t. $\theta$ is explicitly estimated; see Fig.\ref{fig:graph_lambda1}.
\end{itemize}

\begin{figure*}[h]
  \begin{minipage}[c]{0.48\linewidth}
    \centering 
    \includegraphics[scale=0.26]{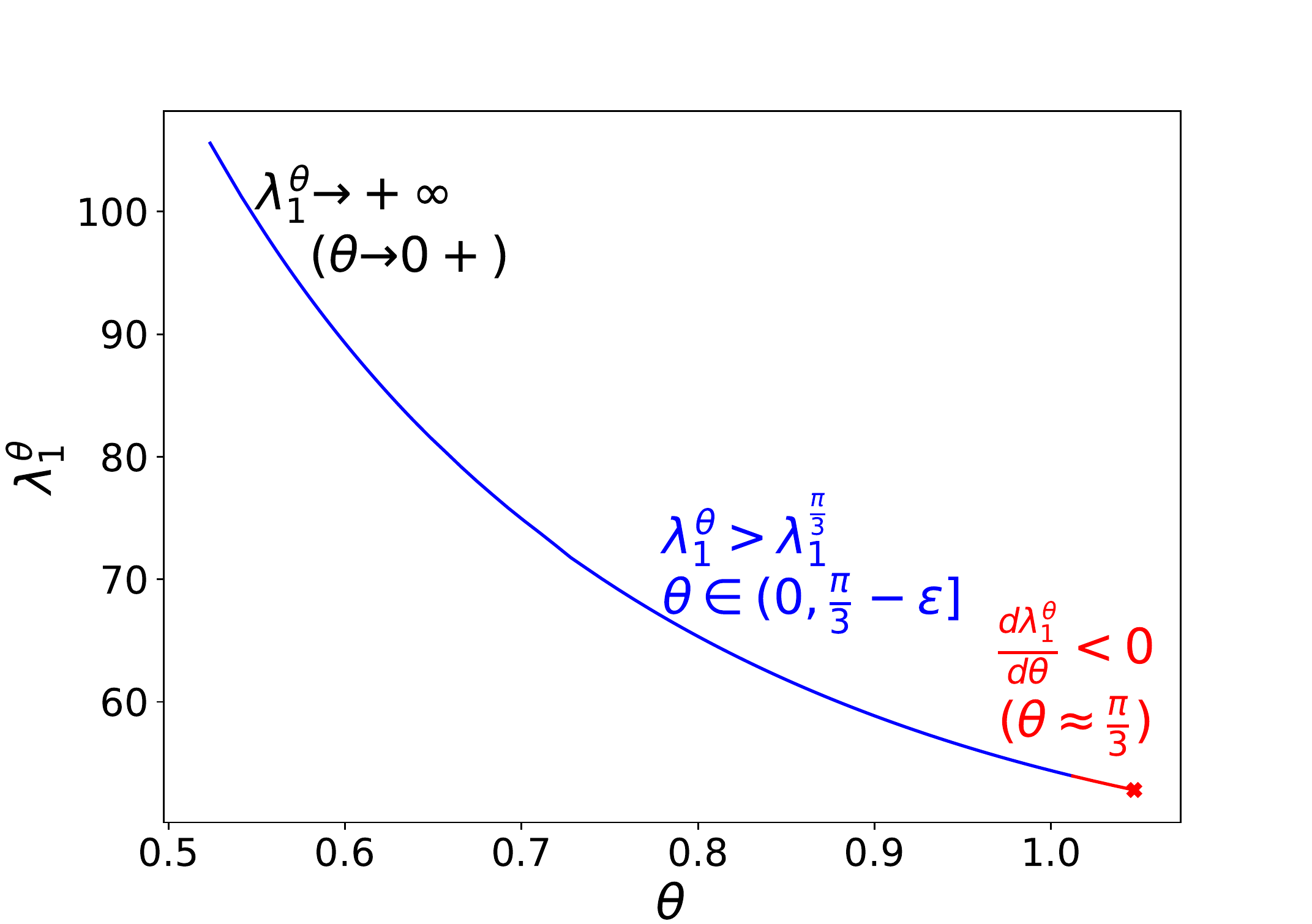}
  \end{minipage}
  \begin{minipage}[c]{0.48\linewidth}
    \centering
    \includegraphics[scale=0.26]{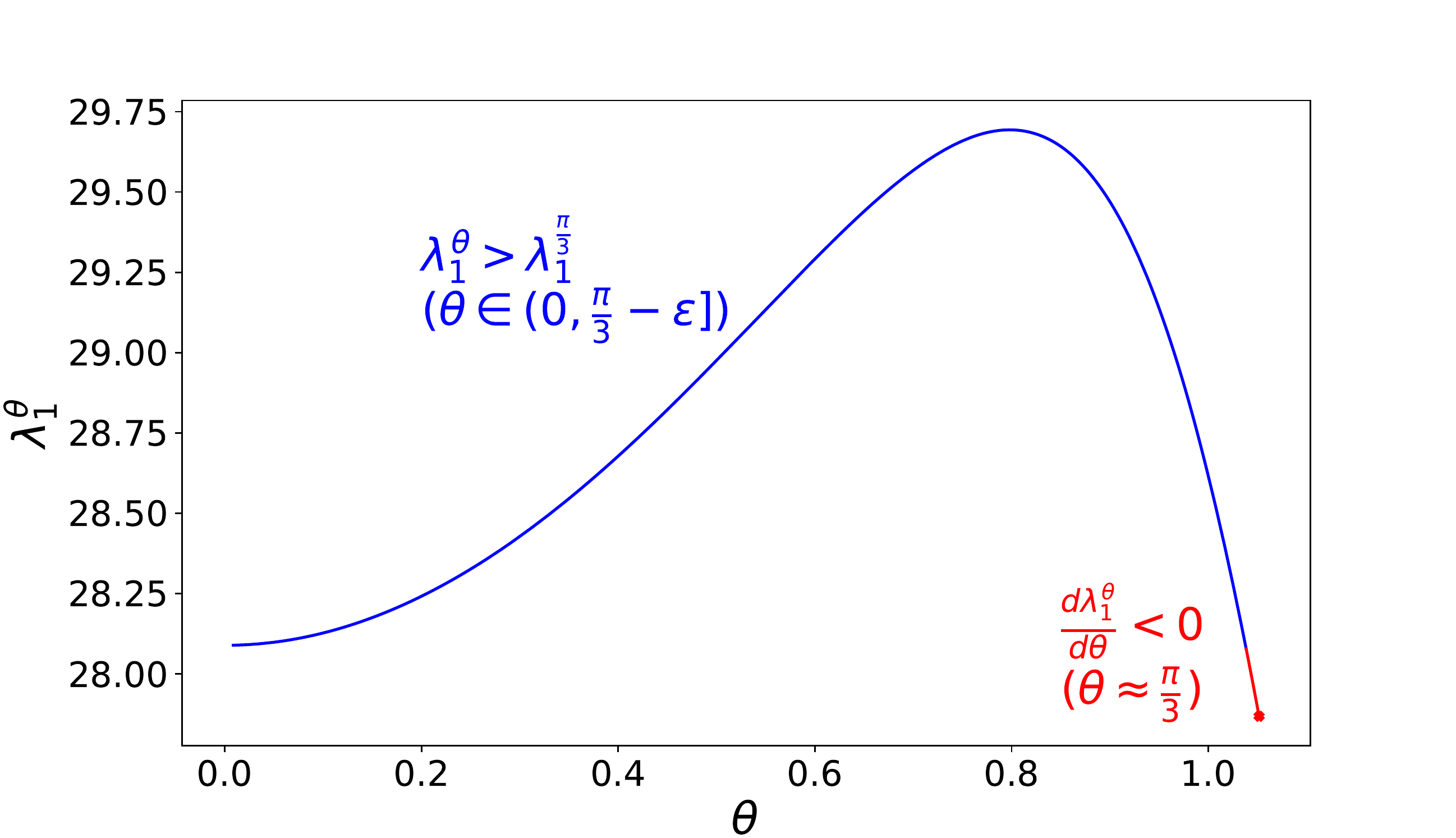}
  \end{minipage}
  \caption{Behavior of $\lambda_1(T^\theta)$ for two eigenvalue problems\\(Left: $V=V_0$; Right: $V=V_e$)}
  \label{fig:graph_lambda1}
\end{figure*}

\subsection{Monotonicity of $\lambda_k$ with respect to $y$-coordinate of $B$}

Let us first introduce a result about the domain monotonicity property of eigenvalues; the details of the proof are provided in Appendix A.
\begin{lemma}\label{lem:vertival_monotonicity}
For fixed $x$-coordinate of $B$, $\lambda_k(T)$ is monotonically decreasing on the $y$-coordinate of vertex $B$.
\end{lemma}

The above lemma implies that the minimum value of $\lambda_1(T)$ is taken when $B$ is on the arc such that $r=1,\theta\in(0,\pi/3]$; see Figure \ref{fig:monotonicity}.
\begin{figure}[ht]
    \centering
    \includegraphics[keepaspectratio, scale=0.35]{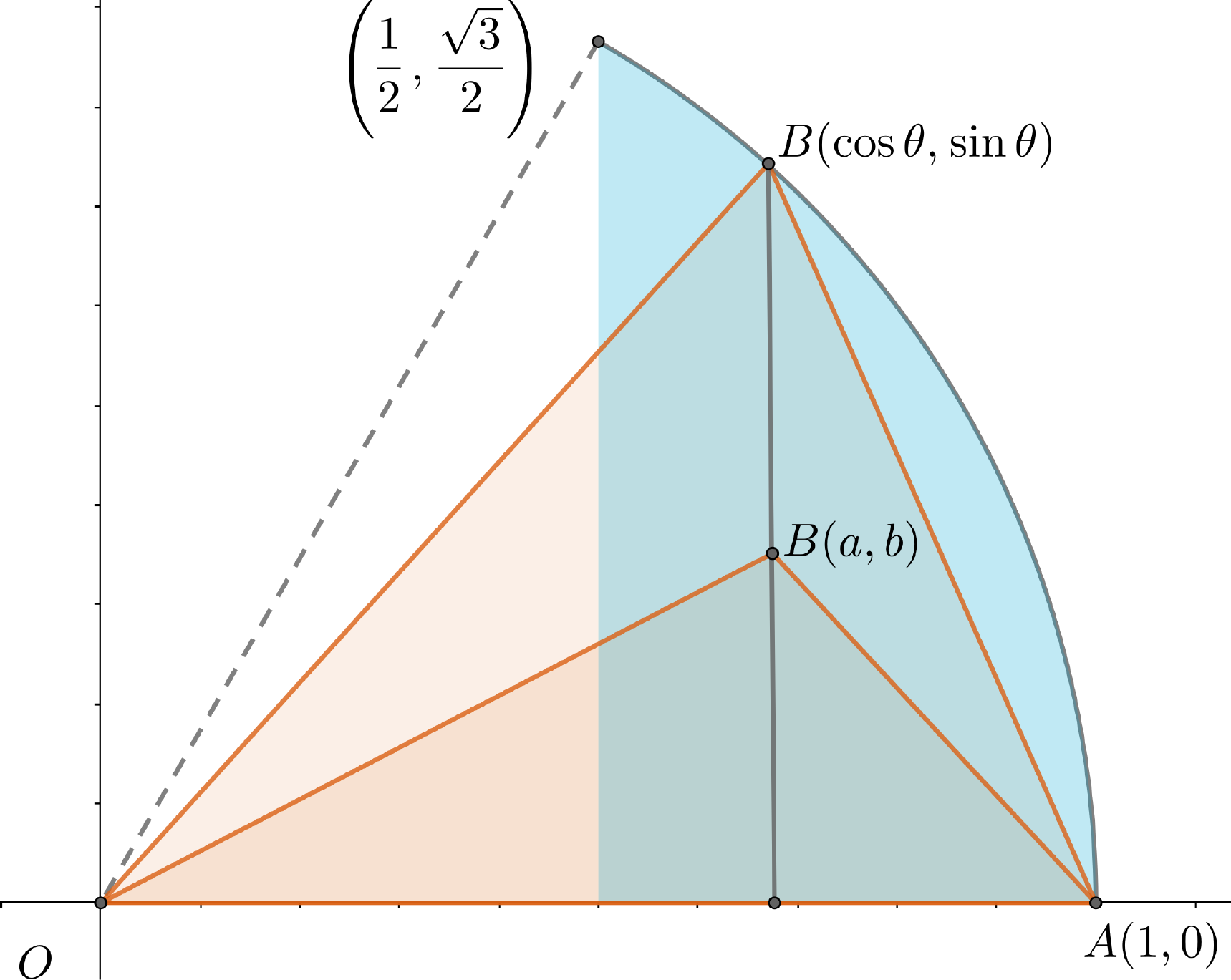}
    \caption{\label{fig:monotonicity} Parameter reduction of the shape optimization problem}
\end{figure}

\subsection{Perturbation of functions with respect to variation of triangles}
Let $T$ be the triangle with  vertices $O(0,0)$, $A(1,0)$ and $B(a,b)$. Let us introduce the perturbation of $T$ by linear transform $\Phi:T\to \widetilde{T}$:
$$
\left(
\begin{array}{c} \tilde{x} \\ \tilde{y} \end{array}
\right)
=Q
\left(
\begin{array}{c} x \\ y \end{array}
\right)
,\quad
Q=\left( \begin{array}{cc}1 &\alpha \\
0 &\beta  \end{array}\right)
\quad (\beta > 0).
$$
For $u$ over $T$, define $\tilde{u}$ over $\widetilde{T}$ by
$\tilde{u}=u\circ\Phi$.
The transpose of $Q$ is denoted by $Q^\intercal$.
Let $\widetilde{\nabla} \tilde{u}:=(\frac{\partial}{\partial \widetilde{x}}\tilde{u},\frac{\partial}{\partial \widetilde{y}}\tilde{u})^\intercal $ be the gradient of $\tilde{u}$.
It holds that $\widetilde{\nabla}\tilde{u}(\tilde x,\tilde y)=Q^{-\intercal}
    \nabla u(x,y)$.
Let  
$\lambda_{\min}(\cdot) $ and $\lambda_{\max}(\cdot)$ denote the minimum and the maximum eigenvalues of a given square symmetric matrix, respectively.

\medskip

Below, let us confirm the properties for the perturbation of $u$; see  \cite{Liao-2019} for a detailed proof. 
\begin{lemma}\label{lemma2} Given function $u$ over $T$, define $\tilde{u}=u\circ Q^{-1}$ over $\widetilde{T}$. 
\begin{enumerate}
\item [(a)]  For $L^2(T)$-norm, we have
$$
\|\tilde u\|^2_{\widetilde T}=\beta\|u\|^2_{T}\:.
$$
\item [(b)]
  For $H^1(T)$-norm, we have
  \begin{equation}
\label{eq:variation-h1-semi-norm}    \displaystyle\lambda_{\min}(QQ^\intercal)\|\nabla\tilde  u\|^2_{\widetilde T}\leq \beta \|\nabla u\|^2_{T}\leq\lambda_{\max}(QQ^\intercal)\|\nabla\tilde u\|^2_{\widetilde T}\:~.
  \end{equation}
  Let $\gamma=\alpha^2+\beta^2+1$. The eigenvalues of $QQ^\intercal$ are given by
  \begin{equation*}
    \lambda_{\min}(QQ^\intercal)=\frac{\gamma-\sqrt{\gamma^2-4\beta^2}}{2},\quad
    \lambda_{\max}(QQ^\intercal)=\frac{\gamma+\sqrt{\gamma^2-4\beta^2}}{2}.
  \end{equation*}

\item [(c)]
   For the quantities involving the first derivative, we have
\begin{equation}
    \label{eq:relation_d1_terms}
\left(
\begin{array}{c}
(\tilde{u}_{\tilde{x}}, \tilde{u}_{\tilde{x}})_{\widetilde{T}}
\\
(\tilde{u}_{\tilde{x}}, \tilde{u}_{\tilde{y}})_{\widetilde{T}}\\
(\tilde{u}_{\tilde{y}}, \tilde{u}_{\tilde{y}})_{\widetilde{T}}
\end{array}
\right)
=\left(
\begin{array}{ccc}
\beta & 0 & 0  \\
-\alpha & 1 & 0 \\
\alpha^2\beta^{-1} & -2\alpha\beta^{-1} & \beta^{-1} 
\end{array}
\right)
\left(
\begin{array}{c}
({u}_{x}, {u}_{x})_T   \\
({u}_{x}, {u}_{y})_T   \\ 
({u}_{y}, {u}_{y})_T  
\end{array}
\right)~.
\end{equation}
\end{enumerate}
\end{lemma}
\begin{proof}
The equality of (a) is evident. 
Since 
\(\nabla u=Q^\intercal \nabla \tilde{u}\), we have \begin{equation*}
\lambda_{\min}(QQ^\intercal)\cdot(\tilde u_{\tilde x}^2+\tilde u_{\tilde y}^2)\le u_x^2+u_y^2\le\lambda_{\max}(QQ^\intercal)\cdot(\tilde u_{\tilde x}^2+\tilde u_{\tilde
{y}}^2)~.
\end{equation*}
Noting that $\text{d}\tilde{x}\text{d}\tilde{y}
= \beta \text{d}x\text{d}y$ holds for the integrates over $T$ and $\widetilde{T}$, we obtain \eqref{eq:variation-h1-semi-norm}.
The relation of \eqref{eq:relation_d1_terms} can be shown with an analogous argument. 

\end{proof}

Next, let us consider a concrete transformation 
 $\Phi_{\theta,\tilde\theta}$ that maps
$T^\theta$ to $T^{\tilde{\theta}}$:  for $\theta,\tilde\theta\in(0,\pi)$, the transformation matrix is given by
\begin{equation*}
    S_{\theta,\tilde\theta}:=
    \begin{pmatrix}
     1 & (\cos\tilde\theta-\cos\theta)/\sin\theta\\
     0 &  \sin\tilde\theta/\sin\theta
    \end{pmatrix}.
\end{equation*}

\begin{lemma}[Eigenvalue perturbation; Extension of Theorem 4.2
 of \cite{liu2015framework}\label{perturbation}]
For $\theta,\tilde\theta\in (0,\pi)$, let $\widetilde {B}(\cos\tilde\theta,\sin\tilde\theta)$ be a perturbation of $B(\cos\theta,\sin\theta)$.

Then, we have
\begin{equation}\label{eq:perturbation}
\min\left\{\frac{\cos\tilde\theta-1}{\cos\theta-1},\frac{\cos\tilde\theta+1}{\cos\theta+1}\right\}
\cdot\lambda^{\theta}_k
\leq
\lambda^{\tilde{\theta}}_k\leq
\max\left\{\frac{\cos\tilde\theta-1}{\cos\theta-1},\frac{\cos\tilde\theta+1}{\cos\theta+1}\right\}
\cdot\lambda^{\theta}_k.
\end{equation}

\end{lemma}
\begin{proof}
This result is an extension of the estimation of Theorem 4.2 in \cite{liu2015framework}, where only the first eigenvalue is considered. The detailed proof for the general $k$-th eigenvalue is provided in the appendix.
\end{proof}

\begin{remark}\label{rem:h-perturbation}
The results in Lemma \ref{perturbation} are also valid for discretized eigenvalues $\lambda_{k,h}^{\CR}$, $\lambda_{k,h}^{\CG}$.
\end{remark}

\medskip

Given eigenvalue $\lambda^\theta_k$, denote by $u^\theta$ an $L^2$-normalized  eigenfunction associated to $\lambda^\theta_k$. 
Note that $u^\theta$ is not uniquely defined due to the sign of the eigenfunction and the multiplicity of the eigenvalue. 

Below, we consider a sequence  $\{\theta_i\}$ which converges to $\theta$.
For the $L^2$ normalized eigenfunction $u^{\theta_i}$, 
the Rayleigh quotient $R(u^{\theta_i})=\|\nabla u^{\theta_i}\|^2_{T^{\theta_i}}=\lambda_{k}^{\theta_i}$ and the continuity of the eigenvalue $\lambda_{k}^{\theta_i}$ w.r.t. $\theta_i$ tell that 
$$
\lim_{i\to\infty}\|\nabla u^{\theta_i}\|^2_{T^{\theta_i}}=\lim_{i\to\infty}\lambda^{\theta_i}_k=\lambda^\theta_k~.
$$
Define 
$\tilde{u}_i:=u^{\theta_i}\circ\Phi_{\theta,\theta_i}$. Since the boundary value condition of $u^\theta_i$ in either $V_0(T^{\theta_i})$ or $V_e(T^{\theta_i})$ is well preserved, we have  $\tilde u_i\in V(T^\theta)$.

Note that the mapping matrix 
$S_{\theta,\theta_i}$ converges to the identity matrix strongly, which leads to the following properties about $\tilde{u}_i$.

\begin{equation}
    \label{eq:unit-norm-of-u0}
    \lim_{i\to\infty}\|\tilde{u}_i\|^2_{T^{\theta}}
    =
\lim_{i\to\infty}\|u^{\theta_i}\|^2_{T^{\theta_i}}
    =1~,
\end{equation}
\begin{equation}
\label{seq_convergence}
\lim_{i\to\infty}\|\nabla \tilde{u}_i \|^2_{T^\theta}=\lim_{i\to\infty}\|\nabla u^{\theta_i}\|^2_{T^{\theta_i}}=\lambda^\theta_k~.
\end{equation}
The convergence (\ref{eq:unit-norm-of-u0}) and (\ref{seq_convergence}) imply the boundedness of the sequence 
$\{\tilde u^i\}_i$ of $V(T^\theta)$ under both $\|\cdot\|_{V(T^\theta)}$ and $\|\cdot\|_{T^\theta}$ norms.





\subsection{Derivative of simple eigenvalue}
To prove that the first eigenvalue $\lambda_1^\theta$ is monotonically decreasing w.r.t. $\theta$ on certain intervals, let us introduce several lemmas that help to estimate the derivative.

\begin{lemma}\label{convergence-existence}
Let $\{\theta_i\}^\infty_{i=1}\subset(0,\pi)$ be a sequence  convergent to $\theta\in(0,\pi)$. 
For any sequence of eigenpairs $\{(\lambda_k^{\theta_i},u^{\theta_i})\}_{i=1}^\infty$, 
there exist a subsequence $\{u^{\theta_{i_j}}\}_{j=1}^\infty$ and an eigenfunction $u_0\in V(T^\theta)$ corresponding to $\lambda_k^\theta$ such that, 
\begin{equation}
\label{eq:lem-congernce-eig-func}
    \|\nabla u_0\|_{T^\theta}^2=\lambda^\theta_k,~\|u_0\|_{T^\theta}=1, \quad    \lim_{j\to\infty}\|\tilde{u}_{i_j}-u_0\|_{V(T^{\theta})}=0~.
\end{equation}
\end{lemma}

\begin{proof}
First, let us list the  properties directly after the boundedness of the sequence 
$\{\tilde u_{i}\}_i$ under $\|\cdot\|_{V(T^\theta)}$ norm.
The Rellich--Kondrashov theorem makes certain the existence of a sub-sequence 
$\{\tilde{u}_{i_j}\}_j$ and  $u_0\in V(T^\theta)$ that have the weak convergence  in $V(T^\theta)$:
\begin{equation}
    \label{weak-convergence}
    \lim_{j\to\infty}(\nabla\tilde{u}_{i_j},\nabla v)_{T^\theta}=(\nabla u_0,\nabla v)_{T^\theta} ~\forall v\in V(T^\theta)~,
\end{equation}
and the strong convergence 
 in $L^2(T^\theta)$:
    \begin{equation}
\lim_{j\to\infty}\|\tilde{u}_{i_j}-u_0\|_{T^\theta}=0~.
\end{equation}
Also, the convergence (\ref{eq:unit-norm-of-u0}) tells that $\lim_{i\to\infty}\|\tilde{u}_{i_j}\|^2_{T^{\theta}}
    =\|u_0\|^2_{T^\theta}=1 $.

Next, we show that $u_0$ is also an  eigenfunction. Take an arbitrary test function $v\in V(T^\theta)$. Let $\tilde{v}^{\theta_i} = v\circ \Phi^{-1}_{\theta, \theta_{i}} \in 
V(T^{\theta_{i_j}})$. 
For eigenfunction  $u^{\theta_{i}} \in V(T^{\theta_{i}})$,  we have
\begin{equation}
\label{eq:eig_at_theta_i}
    (\nabla u^{\theta_{i}},\nabla \tilde  v^{\theta_{i}})_{T^{\theta_{i}}}=\lambda^{\theta_{i}}_k(u^{\theta_{i}},  \tilde v^{\theta_i})_{T^{\theta_{i}}}.
\end{equation}
The strong convergence of $\{S_{\theta,\theta_i}\}$ and the week convergence \eqref{weak-convergence} enable the following limit:
$$
\lim_{j\to \infty} (\nabla u^{\theta_{i_j}},\nabla\tilde{v}^{\theta_{i_j}})_{T^{\theta_{i_j}}} =
\lim_{j\to \infty} (\nabla \tilde{u}_{i_j},\nabla  v)_{T^{\theta}} 
=(\nabla u_0, \nabla v)_{T^{\theta}} ~.
$$
By taking the limit for the two sides of \eqref{eq:eig_at_theta_i}, we have
\begin{equation}\label{eq:convergent-to-eigenfunc}
    (\nabla u_0,\nabla  v)_{T^\theta}=\lambda^{\theta}_k(u_0,v)_{T^\theta}~.
\end{equation}
From the arbitrariness of $v$,  it is clear that $u_0$ is an eigenfunction corresponding to $\lambda^\theta_k$ and 
$\|\nabla u_0\|_{T^\theta}^2=\lambda^\theta_k$.

Finally, we show the strong convergence of $\{\tilde{u}_{i_j}\}_j$ in $V(T^\theta)$. From the strong convergence  (\ref{seq_convergence}) and the equality $\|\nabla u_0\|_{T^\theta}^2=\lambda^\theta_k$, the following convergence about the norm of the sub-sequence is available:
\begin{equation}
    \label{norm-convergence}
    \lim_{j\to\infty} \|
    \nabla\tilde{u}_{i_j}\|^2_{T^\theta}=
    \| \nabla u_0\|_{T^\theta}^2
    =\lambda^\theta_k~.
\end{equation}
The weak convergence (\ref{weak-convergence}) and the convergence of norm (\ref{norm-convergence}) enable the convergence of 
\eqref{eq:lem-congernce-eig-func}.
\end{proof}

\begin{theorem}\label{thm:derivative_formula}
    Let $I(\subset(0,\pi))$ be an open interval such that the $k$-th eigenvalue $\lambda_k^\theta$ is simple for any $\theta\in I$.
    Then we have
    \begin{equation}\label{eq:d-formula}
        \frac{d\lambda^\theta_k}{d\theta}=-2\cot\theta\cdot\|u^{\theta}_{y}\|_{T^\theta}^2+2({u^{\theta}_{x}},{u^{\theta}_{y}})_{T^\theta}~~\mbox{ for 
 any }\theta\in I,
    \end{equation}
    where $u^\theta$ denotes an eigenfunction associated with $\lambda^\theta_k$ such that $\|u^\theta\|_{T^\theta}=1$.
\end{theorem}

\begin{proof}

Let $\theta,\tilde\theta\in I$ with $\theta\neq\tilde\theta$.
Let 
$(\lambda_k^{\theta},u^{\theta})$ and $(\lambda_k^{\tilde\theta},u^{\tilde\theta})$ be two eigenpairs corresponding to ${\theta}$ and ${\tilde{\theta}}$, respectively.
For $(\lambda_k^{\tilde\theta},u^{\tilde\theta})$, the following variational equation holds:
\begin{equation*}
(\nabla u^{\tilde\theta},\nabla\tilde v)_{T^{\tilde\theta}}=\lambda^{\tilde\theta}_k(u^{\tilde\theta},\tilde v)_{T^{\tilde\theta}} \hspace{1em} \forall \tilde v\in V(T^{\tilde\theta}),
\end{equation*}
Introduce $\tilde{u}:=u^{\tilde\theta}\circ\Phi_{\theta,\tilde\theta} \in V(T^\theta)$. Note that
$\widetilde{\nabla} u^{\tilde\theta} = (S^{-\intercal}_{\theta,\tilde\theta})\nabla \tilde{u}$, we have
\begin{equation*}
    \left(
    (S^{-\intercal}_{\theta,\tilde\theta})\nabla \tilde{u}
    ,
    (S^{-\intercal}_{\theta,\tilde\theta})\nabla v
    \right)_{T^\theta}
    =
    \lambda^{\tilde\theta}_k
    (\tilde{u},v)_{T^\theta}
    \hspace{0.5em}\forall v\in V(T^\theta)~.
\end{equation*}
Letting $P_{\theta,\tilde\theta}:=S^{-1}_{\theta,\tilde\theta}S^{-\intercal}_{\theta,\tilde\theta}$ and 
substituting $v=u^\theta$, 
\begin{equation}
\label{perturbated-formula}
    \left(
    P_{\theta,\tilde\theta}\nabla \tilde{u}
    ,
\nabla u^\theta
    \right)_{T^\theta}
    =
    \lambda^{\tilde\theta}_k
    (\tilde{u},u^\theta)_{T^\theta}.
\end{equation}
Also, for $(\lambda_k^{\theta}, u^\theta)$, it holds that
\begin{equation*}
(\nabla u^{\theta},\nabla v)_{T^{\theta}}=\lambda^\theta_k(u^\theta,v)_{T^{\theta}} \hspace{1em} \forall v\in V(T^\theta).
\end{equation*}
Take $v=\tilde{u}$ in the above variation equation, then we have
\begin{equation}\label{base-formula}
(\nabla u^{\theta},\nabla \tilde{u})_{T^{\theta}}=\lambda^\theta_k(u^\theta,\tilde{u})_{T^{\theta}}.
\end{equation}
From (\ref{perturbated-formula}) and (\ref{base-formula}), we have
\begin{align*}
    \frac{\lambda_k^{\tilde{\theta}}-\lambda_k^\theta}{\tilde{\theta}-\theta}
    &=
    \frac{\left(
    \frac{1}{\tilde\theta-\theta}(P_{\theta,\tilde\theta}-I)\nabla \tilde{u}
    ,
    \nabla u^{\theta}
    \right)_{T^\theta}}
    {(\tilde{u}, u^{\theta})_{T^\theta}}=:\zeta(\tilde{\theta})~.
\end{align*}

Next, for a fixed $\theta$ along with an eigenpair $(\lambda_k^\theta, u^\theta)$, let us show the existence of the limit of the above equation by confirming the convergence of $\zeta({\theta_i})$ for arbitrary sequence $\{\theta_i\} \subset I$ that converges to $\theta$.
%
%
%
%
Below, we show two properties of the sequence  $\{\zeta(\theta_i)\}$ .
\paragraph{(a) Boundedness of  $\{\zeta(\theta_i)\}$} 
Let ${u}_i:=u^{\theta_i}\circ\Phi_{\theta,\theta_{i}}\in V(T^\theta)$. First, we show that 
$\lim_{i\to\infty}|( u_i,u^\theta)_{T^\theta}|>0$ by contradiction.
From Lemma \ref{convergence-existence}, there exists a sub-sequence $\{ u_{i_j}\}_{j=1}^\infty$ of $\{ u_i\}_{i=1}^\infty$ and an eigenfunction $u_0\in V(T^\theta)$ corresponding to $\lambda_k^\theta$ such that
$$  \|\nabla u_0\|_{T^\theta}^2=\lambda^\theta_k,~\|u_0\|_{T^\theta}=1, \quad    \lim_{j\to\infty}\|\nabla u_{i_j}-\nabla u_0\|_{T^{\theta}}=0~.
$$
Due to the simplicity of $\lambda_k^\theta$, we have $u_0=\pm u^\theta$. 
If
$\lim_{i\to\infty}|( u_i,u^\theta)_{T^\theta}|=0$, then
$
    |(u_0,u^\theta)_{T^\theta}|=\lim_{j\to\infty}|(u_{i_j},u^\theta)_{T^\theta}|=0,
$
which implies $u_0$ is orthogonal to $u^\theta$ and contradicts the simplicity of $\lambda_k^\theta$. 
Also, note that \begin{equation}\label{eq:convergence-of-P}
    \lim_{i\to\infty}
    \frac{P_{\theta,\theta_{i}}-I}{\theta_{i}-\theta}
    =
    \lim_{i\to\infty}
    \frac{1}{\theta_i-\theta}
    \begin{pmatrix}
        \frac{(\cos\theta_i-\cos\theta)^2}{\sin^2\theta_i} & \frac{\sin\theta(\cos\theta-\cos\theta_i)}{\sin^2\theta_i}\\
        \frac{\sin\theta(\cos\theta-\cos\theta_i)}{\sin^2\theta_i} & \frac{\sin\theta^2-\sin\theta^2_i}{\sin^2\theta_i}
    \end{pmatrix}\\
    =
    \begin{pmatrix}
    0 & 1\\
    1 & -2\cot\theta
    \end{pmatrix}.
\end{equation}
By further applying the Schwarz inequalities, it is confirmed that $\{\zeta(\theta_i)\}$ is bounded. 

\paragraph{(b) Convergence of sequence $\{\zeta(\theta_i)\}$ }

Let $\{\hat\theta_i\}$ be a sub-sequence of $\{\theta_i\}$ such that  $\lim_{i\to\infty}\zeta(\hat\theta_i)$ exists.
From Lemma \ref{convergence-existence}, there exist a sub-sequence $\{\hat\theta_{i_j}\}_{j=1}^\infty$ of $\{\hat\theta_i\}_{i=1}^\infty$ and an eigenfunction $u_0(=\pm u^\theta)$ corresponding to $\lambda_k^\theta$ such that
\begin{gather*}
    \lim_{j\to\infty}\|\nabla (u^{\hat\theta_{i_j}}\circ\Phi_{\theta,\hat\theta_{i_j}})-\nabla u_0\|_{T^{\theta}}=0,~~
    \|u_0\|_{T^\theta}=1.
\end{gather*}
Thus, it holds that
\begin{equation*}
(    \lim_{i\to\infty}\zeta(\hat\theta_{i})
    =
)    \lim_{j\to\infty}\zeta(\hat\theta_{i_j})
    =
    -2\cot\theta\cdot
    \|u^{\theta}_{y}\|_{T^\theta}^2
    +
    2(u^{\theta}_{x},u^{\theta}_y)_{T^\theta}~.
\end{equation*}
Since $\lambda_k^\theta$ is simple, the above limit is uniquely determined and is independent of the choice of a sub-sequence $\{\hat\theta_{i}\}$.
As any convergent sub-sequence of $\{\zeta(\hat\theta_{i})\}$ has the same limit, it follows that
\begin{equation*}
    \lim_{i\to\infty}\zeta(\theta_{i})
    =
    -2\cot\theta\cdot
    \|u^{\theta}_{y}\|_{T^\theta}^2
    +
    2(u^{\theta}_{x},u^{\theta}_{y})_{T^\theta}.
\end{equation*}

From the above two properties, one can draw the conclusion of \eqref{eq:d-formula}.
\end{proof}


\medskip

\begin{remark}
Hadamard considered the derivative of eigenvalues with respect to domain shape variation in the early days \cite{hadamard1908memoire}. In classical literature, most discussion focuses on homogeneous Dirichlet eigenvalue problem on domains with smooth boundaries (see \cite{hadamard1908memoire, rousselet1983shape}). 
Hadamard's method can be in principle applied to the non-homogeneous Neumann eigenvalue problem over polygonal domains, while a detailed argument about the boundary condition is needed. Here, Theorem \ref{thm:derivative_formula} provides an elementary and concise proof for the derivative of eigenvalues.

\end{remark}

For $u\in V(T^\theta)$, let us introduce a functional $F$ to simplify the notation: 
\begin{equation}
    F(u):=-2\cot\theta\cdot\frac{\|u_{y}\|_{T^\theta}^2}{\|u\|_{T^\theta}^2}+2\frac{(u_{x},u_{y})_{T^\theta}}{\|u\|_{T^\theta}^2}.
\end{equation}
For $\theta\in(0,\pi)$, and $a,b,c\in\mathbb{R}$ with $a\leq b< c$, introduce quantity $\mbox{Err}(a,b,c,\theta)$ such that
\begin{equation}
    \mbox{Err}(a,b,c,\theta):=2\sqrt{b}\left(\sqrt{2}\cot\theta+2\right)\eta(a,b,c)~.
\end{equation}

The following theorem discusses the estimation of $F(u)$ using FEM approximations.

\begin{theorem}\label{thm:d-estimation}
Let $I(\subset(0,\pi))$ be an open interval such that the first eigenvalue $\lambda_1^\theta$ is simple for any $\theta\in I$. Let $u^\theta$ be an 
$L^2$-normalized eigenfunction corresponding to $\lambda_1^\theta$. Let $u^{\theta}_h\in V_h^{CG}(T^\theta)$ be an 
$L^2$-normalized approximate eigenfunciton of $u^\theta$. Take $\lambda_{1,h}^{\theta}:=\|\nabla u^{\theta}_h\|^2_{T^\theta}$.
Let $\rho^\theta\in\mathbb{R}$ be a quantity such that $(\lambda_1^\theta\leq)\lambda_{1,h}^{\theta}<\rho^\theta\leq\lambda_2^\theta$.
Then we have
\begin{equation*}
    \left|\frac{d\lambda_1^\theta}{d\theta}-F(u_h^{\theta})\right|\\
    \leq
    \mbox{Err}(\lambda_1^\theta,\lambda_{1,h}^{\theta},\rho^\theta,\theta       )~.
\end{equation*}
\end{theorem}

\begin{proof}
Note that the value of $F(u)$ does not change if the sign of $u$ is changed.
From Lemma \ref{lem:original_eigenvec_estimation}, with a proper selection of the sign of $u_h^\theta$, we have
\begin{equation*}
    \|\nabla u^\theta-\nabla u^{\theta}_h\|_{T^\theta}\leq\eta(\lambda_1^\theta,\lambda_{1,h}^{\theta},\rho^\theta).
\end{equation*}
Recall the derivative of $\lambda_1^\theta$ obtained in Theorem \ref{thm:derivative_formula} such that
\begin{equation*}
    \frac{d\lambda^\theta_1}{d\theta}=F(u^\theta)=-2\cot\theta\cdot\|u^{\theta}_{y}\|_{T^\theta}^2+2({u_{x}^{\theta}},{u_{y}^{\theta}})_{T^\theta}~.
\end{equation*}
Let us compare the difference of the  terms of $F(u^\theta)$ and $F(u_h^\theta)$:
\begin{align*}
   & \left|(u^\theta_x,u^\theta_y)_{T^\theta}-(u^{\theta}_{h,x},u^{\theta}_{h,y})_{T^\theta}\right|\\
     \leq &
    |(u^\theta_x,u^\theta_y)_{T^\theta}-(u^{\theta}_{h,x},u^\theta_y)_{T^\theta}|
    +
    |(u^{\theta}_{h,x},u^\theta_y)_{T^\theta}-(u^{\theta}_{h,x},u^{\theta}_{h,y})_{T^\theta}|\\
    \leq &
    \left\|u^\theta_x-u^{\theta}_{h,x}\right\|_{T^\theta}\left\|u^\theta_y\right\|_{T^\theta}+\left\|u^{\theta}_{h,x}\right\|_{T^\theta}\left\|u^\theta_y-u^{\theta}_{h,y}\right\|_{T^\theta}\\
    \leq &
    \sqrt{\lambda_{1}^\theta+\lambda_{1,h}^\theta}\|\nabla u^\theta-\nabla u^{\theta}_h\|_{T^\theta}
    \end{align*}
and
    \begin{align*}
&\left|\|u_y^{\theta}\|_{T^\theta}^2-\|u^{\theta}_{h,y}\|_{T^\theta}^2\right| \\
    \leq &   \left\|u_y^{\theta}+u^{\theta}_{h,y}\right\|_{T^\theta}\cdot
    \left\|u_y^{\theta}-u^{\theta}_{h,y}\right\|_{T^\theta}\\
    \leq&\left(\|u_y^{\theta}\|_{T^\theta}+\|u^{\theta}_{h,y}\|_{T^\theta}\right)
    \left\|\nabla u^{\theta}-\nabla u^{\theta}_h\right\|_{T^\theta}\\
    \leq &
    (\sqrt{\lambda_{1}^\theta}+\sqrt{\lambda_{1,h}^\theta})\|\nabla u^\theta-\nabla u^{\theta}_h\|_{T^\theta}.
\end{align*}
Using the fact $\lambda_1^\theta\leq\lambda_{1,h}^\theta$, we have
\begin{align*}
&\left|\frac{d\lambda_1^\theta}{d\theta}-F(u^{\theta}_h)\right|\\
    &\leq \frac{2}{\tan\theta}\left|\|u^{\theta}_y\|_{T^\theta}^2-\|u^{\theta}_{h,y}\|_{T^\theta}^2\right|
    +2\left|(u^\theta_x,u^\theta_y)_{T^\theta}-(u^{\theta}_{h,x},u^{\theta}_{h,y})_{T^\theta}\right|\\
    & \leq
    2\sqrt{\lambda_{1,h}^\theta}\left(\sqrt{2}\cot\theta+2\right)\eta(\lambda_1^\theta,\lambda_{1,h}^{\theta},\rho^\theta)
    (=\mbox{Err}(\lambda_1^\theta,\lambda_{1,h}^{\theta},\rho^\theta,\theta)).
\end{align*}
\end{proof}

To obtain the range of $F(u^\theta)(=\frac{d\lambda_1^\theta}{d\theta})$
over an interval $I=[\theta_1,\theta_2]$,  the following estimation is required.
\begin{equation}
\label{eq:range_of_f_over_interval}
\sup_{{\theta}\in [\theta_1, \theta_2] }
\left|\frac{d\lambda_1^\theta}{d\theta}-F(u_h^{\theta})\right|\\
    \leq
\sup_{{\theta}\in [\theta_1, \theta_2] }
\mbox{Err}(\lambda_1^\theta,\lambda_{1,h}^{\theta},\rho^\theta,\theta)~.
\end{equation}
In evaluating $\mbox{Err}(\lambda_1^\theta,\lambda_{1,h}^{\theta},\rho^\theta,\theta)$, we take $(\lambda_{1,h}^{\theta_1}, u_{h}^{\theta_1})$ as the reference eigenpair and 
${u}^\theta_h:=u_h^{\theta_1} \circ\Phi_{\theta,\theta_1}\in V(T^\theta)$ 
as the approximation to $u^\theta$. Since there are infinite $\theta$'s in interval $I$, the value of $\lambda_1^\theta$ and $\lambda_{1,h}^\theta=R({u}^\theta_h)$ is not evaluated directly, but estimated using the perturbation estimation of Lemma \ref{perturbation}.

\section{Solution to the optimization problems for Laplacian eigenvalues upon triangle shapes}\label{section:numerical-results}

For a triangular domain $T$, $\mathcal{T}^h$ denotes a regular subdivision of $T$ with triangular domains; that is, any two edges $e_i$ and $e_j$ of elements of $\mathcal{T}^h$ satisfy $e_i\cap e_j=e_i=e_j$ or $\mu(e_i\cap e_j)=0$, where $\mu(\cdot)$ is the $1$-dimensional measure.

Let us introduce the finite element spaces $V_h^{\CG}$ and $V_h^{\CR}$ over $\mathcal{T}^h$.
\begin{itemize}
    \item The Lagrange FEM space $V_h^\CG$:
    \begin{equation}
          \label{def:fem-space-cg}
V_h^{\CG}:=
\{v_h~|~v_h\mbox{ is a continuous piecewise linear polynomial on } \mathcal{T}^h.
\}
    \end{equation}
    \item The Crouzeix--Raviart FEM space $V_h^\CR$:
\begin{gather}
\label{def:fem-sapce-cr}
V_h^{\CR}:=\{v_h~|~v_h\mbox{ is a piecewise linear polynomial on } \mathcal{T}^h; \quad\quad \quad\quad \quad\\
\quad \quad \quad v_h\mbox{ is continuous on the midpoint of each inter-element edge } e. \}. \quad\quad \notag
\end{gather}
\end{itemize}
In the numerical results reported in the following two subsections, the FEM spaces are set up over a  uniform triangulation of the triangle domain. Denote by $N$ the subdivision number of the triangulation along the base edge. Below is the detailed setting for the FEM spaces:
$$
V_h^{\CG}: {N}=96,~
\mbox{DOF}  = 4753;\quad
V_h^{\CR}: N= 64,~ \mbox{DOF}= 6240~.
$$

Following the proof outline described at the beginning of Section \ref{section:main-theories}, rigorous calculations are conducted to solve the shape optimization problems for the Dirichlet eigenvalue and the Crouzeix--Raviart interpolation error constant. Below, we describe the algorithms for the calculations needed in Step 2 and Step 3. The concrete parameter settings for the algorithms are stated in each problem at  \S\ref{subsec:dirichelt} and \S\ref{subsec:CR-constant}.

\vspace{1em}

To show $\lambda_1^\theta>\lambda_1^\frac{\pi}{3}$ over $I:=(0,\frac{\pi}{3}-\varepsilon]$, a lower bound of $\lambda_1^\theta$ for all $\theta\in I$ is estimated by utilizing Algorithm \ref{algorithm-1-new}, where the interval $I$ is divided into small sub-intervals $I_i:=(\theta_{i-1},\theta_i]~(i=1,\cdots,N_1)$. 
    \begin{algorithm}[H]
    \caption{\label{algorithm-1-new} Lower bound of $\lambda_1^\theta$ over interval $I$}
     \KwData{Interval $I =\cup I_i$, $I_i=(\theta_{i-1},\theta_{i}]$ ($i=1,\cdots,N_1$)}
     \KwResult{$\underline{\lambda_1^I}$ as a lower bound of ${\lambda^\theta_1}$ over $I$} 
     {\bf Procedure:} 
     For each $I_i$, $i=1, \cdots , N_1$,
     \begin{enumerate}
         \item [1.] Evaluate $\lambda_1^{\theta_i}$ for triangular domain $T^{\theta_i}$ by Lemma \ref{lem:l-estimation}.
         \vspace{0.2em}
        \item [2.] Evaluate $\underline{\lambda_1^{I_i}}$  as a lower bound of ${\lambda^\theta_1}$ over $I_i$ by letting {$\theta=\theta_{i}$ and $\tilde\theta=  \theta_{i-1}$}
 in Lemma \ref{perturbation}
     \end{enumerate}
    The output $\underline{\lambda_1^I}$ is calculated by the minimum value of all $\underline{\lambda_1^{I_i}}$.
    \end{algorithm}

The monotonicity of $\lambda_1^\theta$ w.r.t. $\theta$ over $J:=[\frac{\pi}{3}-\varepsilon,\frac{\pi}{3}]$ is validated through Algorithm \ref{algorithm-2-new}, where the range of $F(u^\theta)$ is rigorously estimated over an equal subdivision of $J$: $J=\cup J_i$,  $J_i:=[\theta_{i},\theta_{i}+h_\theta$, $\theta_i=\pi/3-\varepsilon+(i-1)h_\theta~$, $h_\theta=\varepsilon/N_2$ $(i=1,\cdots,N_2)$. 

\begin{algorithm}[H]
\caption{\label{algorithm-2-new} Range of $F(u^\theta)$ over interval $J$}
 \KwData{Interval $J =\cup J_i$, $J_i=[\theta_{i},\theta_{i+1}]$ ($i=1,\cdots,N_2$)}
 \KwResult{[$\underline{F}, \overline{F}$] as the estimation of range of $F(u^\theta)$ over $J$} 
 {\bf Procedure:} 
 For each $J_i$, $i=1, \cdots , N_2$,
 \begin{enumerate}
     \item [1.] Evaluate $\lambda_1^{\theta_i}$, $\lambda_2^{\theta_i}$ for triangular domain $T^{\theta_i}$ by Lemma \ref{lem:l-estimation}.
    \item [2.] Evaluate the range of $F(u^{\theta})$ over $J_i$ by Theorem \ref{thm:d-estimation} and the estimation \eqref{eq:range_of_f_over_interval}.
 \end{enumerate}
The output $[\underline{F}, \overline{F}]$ is calculated by the span of the range of $F(u^\theta)$ over each $J_i$.
\end{algorithm}

\begin{remark}
    Before applying Theorem \ref{thm:d-estimation} in Procedure 2, the simplicity of $\lambda_1^\theta$ over $J$ needs to be checked. While the simplicity of $\lambda_1^\theta$ is a known fact in the case of the Dirichlet boundary condition $V(T)=V_0(T)$ (see \cite{evans2010partial}), it is not obvious in the case of $V(T)=V_e(T)$. To guarantee the simplicity of $\lambda_1^\theta$ over $J$ in the case of $V(T)=V_e(T)$, the range of $\lambda_1^\theta$ and $\lambda_2^\theta$ are rigorously estimated utilizing Theorem \ref{lem:l-estimation} and \ref{perturbation}. Indeed, we have $\lambda_1^\theta\leq 28.069< 45.011\leq\lambda_2^\theta$ for $\theta\in J$. 
\end{remark}

\begin{remark}
    The evaluation of  $\eta(\lambda_1^\theta,\lambda_{1,h}^{\theta},\rho^\theta)$ over an interval $J_i$ (to be used in the estimation \eqref{eq:range_of_f_over_interval}) can be simplified by utilizing the monotonicity of $\eta$ w.r.t. its parameters. 
        In case $V(T)=V_0(T)$,  
 the quantities $\lambda_1^\theta$,  $\lambda_{1,h}^{\theta}$ and $\rho^\theta$ have the range as $\rho^\theta\in[121,123]$, $\lambda_1^\theta,\lambda_{1,h}^{\theta}\in[52,55]$ over the interval $(\pi/3-\varepsilon, \pi/3]$. In case $V(T)=V_e(T)$, $\rho^\theta\in[43,47]$, $\lambda_1^\theta,\lambda_{1,h}^{\theta}\in[27,29]$. In either case, the real-valued function $\eta(\lambda_1^\theta,\lambda_{1,h}^{\theta},\rho^\theta)$ defined in (\ref{eq:def-eta})
    is monotonically decreasing w.r.t. $\lambda_1^\theta$ and $\rho^\theta$, and increasing w.r.t. $\lambda_{1,h}^{\theta}$.  Therefore, for each sub-interval $J_i$,  a uniform lower bound of $\lambda_1^\theta$ and $\rho^\theta$,  a uniform upper bound of $\lambda_{1,h}^{\theta}$  are sufficient to provide an upper bound of $\eta(\lambda_1^\theta,\lambda_{1,h}^{\theta},\rho^\theta)$ for all $\theta \in J_i$.
\end{remark}
\begin{remark}
    From the inequality (\ref{eq:eig-vec-est}), sharper upper and lower bounds estimations of $\lambda_1$ are needed to decrease the value of $\max_{\theta\in J_i}\eta(\underline{\lambda_1^\theta},\lambda_{1,h}^\theta,\rho^\theta)$ for each $i$. To obtain sufficiently sharp upper and lower bounds, the interval $J$ is divided into $N_2$ subintervals.
\end{remark}

\subsection{Dirichlet eigenvalue problem}
\label{subsec:dirichelt}
Let us consider the case of $V(T)=V_0(T)$ in this section.
Finite element spaces $V_{h,0}^{\CG}$ and $V_{h,0}^{\CR}$ are defined by
\begin{gather*}
V_{h,0}^{\CG} := H^1_0(T) \cap V_{h}^{\CG},\\
V_{h,0}^{\CR}:=\{v_h \in V_h^\CR~: \int_{e} v_h =0 ~ \mbox{ for each boundary edge }e \mbox{ of } \mathcal{T}^h\}~.
\end{gather*}
These spaces are utilized as
$V_{h}^{\CG}$ and $V_{h}^{\CR}$ in Lemma \ref{lem:l-estimation} to estimate bounds of Dirichlet eigenvalues $\lambda_k^\theta$.
The intervals $I=(0,\pi/3-\varepsilon]$, $J=[\pi/3-\varepsilon,\pi/3]$ and their subdivision are selected as follows:
 $$
 \varepsilon=\pi/1500,~ N_1=170, ~ N_2=200~.
 $$
The subdivision nodes of 
 $I_i=(\theta_{i-1},\theta_i]~(i=1,\cdots,N_1)$ are defined by
        \begin{equation*}
           \theta_i=\frac{\pi}{3}*
           \begin{cases}
           0.02*i &\ \mbox{ for }i=0,..,45\\
           0.9+10^{-2}(i-45) &\ \mbox{ for }i=46,\cdots,50\\
           0.95+10^{-3}(i-50) &\ \mbox{ for }i=51,\cdots,90\\
           0.99+10^{-4}(i-90) &\ \mbox{ for }i=91,\cdots,170
           \end{cases}.
        \end{equation*}
    Utilizing Algorithm \ref{algorithm-1-new}, it is confirmed that $\max_{\theta \in I} \lambda_1^\theta>\lambda_1^{\frac{\pi}{3}}$ from the rigorous computation result:
    $$
    52.654\leq\lambda_1^\theta \mbox{ for } \theta\in I, \mbox{ while } \lambda^{\frac{\pi}{3}}_1\leq 52.641.
    $$ 
   
The range of $F(u^\theta)$ for $\theta \in J=[\frac{\pi}{3}-\varepsilon,\pi/3]$ is evaluated over an equal subdivision of $J$ with subdivision number as $N_2=200$.
The obtained range is displayed in  Table \ref{table:problem-1-quantities}. Meanwhile, to confirm 
the affection of terms involved in evaluating the range of $F$, the values of several quantities at $\theta=\pi/3$ are also provided in Table \ref{table:problem-1-quantities}.

\begin{table}[H]
  \caption{\label{table:problem-1-quantities}  The obtained range of $F(u^\theta)$ and related quantities}
  \centering
  \begin{tabular}{cl}
    \hline
    \rule[-2mm]{0mm}{6mm}{}
    $[\underline{F}, \overline{F}]$  & $[-42.461,-18.610]$\\
        \hline
    \rule[-2mm]{0mm}{6mm}{}
        $F(u_h^{\frac{\pi}{3}})$ & $\approx -30.401$ \\
            \rule[-2mm]{0mm}{6mm}{}
    $\mbox{Err}(\lambda_1^\frac{\pi}{3},\lambda_{1,h}^\frac{\pi}{3},\rho^\frac{\pi}{3},\frac{\pi}{3})$  & $\leq 11.791$\\
                \rule[-2mm]{0mm}{6mm}{}
    $\eta(\lambda_1^{\frac{\pi}{3}},\lambda_{1,h}^{\frac{\pi}{3}},\rho^{\frac{\pi}{3}})$ & $\leq 0.3163$\\
            \rule[-2mm]{0mm}{6mm}{}
            $\|u^\frac{\pi}{3}_{h,x}\|_{T^\frac{\pi}{3}}^2$  & $\approx 26.328$\\
                \rule[-2mm]{0mm}{6mm}{}
                $\|u^\frac{\pi}{3}_{h,y}\|_{T^{\frac{\pi}{3}}}^2$  & $\approx 26.328$\\
                    \rule[-2mm]{0mm}{6mm}{}
                    $|(u^\frac{\pi}{3}_{h,x},u^\frac{\pi}{3}_{h,y})_{T^\frac{\pi}{3}}|$  & $\le 3.049\cdot 10^{-10}$\\
                \rule[-2mm]{0mm}{6mm}{}
            $\lambda_{1,h}^{\frac{\pi}{3}}$  & $\approx 52.656$\\
    $\rho^{\frac{\pi}{3}}$  & $= 122.6133$\\
    \hline
  \end{tabular}
\end{table}

\vspace{1em}
Therefore, we draw the conclusion that the regular triangle minimizes the first Dirichlet eigenvalue $\lambda_1(T)$ among all the triangles with the same diameter.

\subsection{Eigenvalue problem for the Crouzeix--Raviart interpolation error constant}
\label{subsec:CR-constant}
To solve the optimization problem for the Crouzeix--Raviart interpolation error constant, we take $V(T)=V_e(T)$ and 
finite element spaces $V_{h,e}^{\CG}$ and $V_{h,e}^{\CR}$ as follows.
\begin{gather*}
V_{h,e}^{\CG} :=\{v_h \in V_h^\CG~: \int_{e_i} v_h =0 ~ \mbox{ for each edge }e_i \mbox{ of }T\}~,\\
V_{h,e}^{\CR}:=\{v_h \in V_h^\CR~: \int_{e_i} v_h =0 ~ \mbox{ for each edge }e_i \mbox{ of }T\}~.
\end{gather*}
These spaces are utilized as $V_{h}^{\CG}$ and $V_{h}^{\CR}$ in Lemma \ref{lem:l-estimation} to estimate bounds of the eigenvalues $\lambda_k^\theta$. The intervals $I=(0,\pi/3-\varepsilon]$, $J=[\pi/3-\varepsilon,\pi/3]$ and their subdivision are selected as follows:
$$
 \varepsilon=\pi/3000,~ N_1=320,  ~N_2 = 100~.
$$
The subdivision nodes of $I_i=(\theta_{i-1},\theta_i]~(i=1,\cdots,N_1)$ are defined by
\begin{equation*}
   \theta_i=\frac{\pi}{3}*
   \begin{cases}
   0.02*i &\ \mbox{ for }i=0,..,40\\
   0.8+10^{-3}(i-40) &\ \mbox{ for }i=41,\cdots,230\\
   0.99+10^{-4}(i-230) &\ \mbox{ for }i=231,\cdots,320
   \end{cases}.
\end{equation*}
Utilizing Algorithm \ref{algorithm-1-new}, it is confirmed that $\max_{\theta \in I} \lambda_1^\theta>\lambda_1^{\frac{\pi}{3}}$ from the rigorous computation result:
    $$
    27.949\leq\lambda_1^\theta \mbox{ for } \theta\in I, \mbox{ while } \lambda^{\frac{\pi}{3}}_1\leq 27.945.
    $$ 
    
The obtained range of $F(u^\theta)$ for $\theta \in J=[\frac{\pi}{3}-\varepsilon,\pi/3]$ and related quantities are listed in Table \ref{table:problem-2-quantities}.

\begin{table}[H]
  \caption{\label{table:problem-2-quantities} The obtained range of $F(u^\theta)$ and related quantities}
  \centering
  \begin{tabular}{cc}
    \hline
    $[\underline{F}, \overline{F}]$  & $[-20.536,-11.711]$\\
    $F(u_h^{\frac{\pi}{3}})$ & $\approx -16.134 $ \\
    $\mbox{Err}(\lambda_1^\frac{\pi}{3},\lambda_{1,h}^\frac{\pi}{3},\rho^\frac{\pi}{3},\frac{\pi}{3})$  & $\leq 4.401$\\
    $\eta(\lambda_1^{\frac{\pi}{3}},\lambda_{1,h}^{\frac{\pi}{3}},\rho^{\frac{\pi}{3}})$ & $\leq 0.1621$\\    
    $\|u^\frac{\pi}{3}_{h,x}\|_{T^\frac{\pi}{3}}^2$  & $\approx 13.973$\\
    $\|u^\frac{\pi}{3}_{h,y}\|_{T^{\frac{\pi}{3}}}^2$  & $\approx 13.973$\\
    $|(u^\frac{\pi}{3}_{h,x},u^\frac{\pi}{3}_{h,y})_{T^\frac{\pi}{3}}|$  & $\le 2.050\cdot 10^{-10}$\\
    $\lambda_{1,h}^{\frac{\pi}{3}}$  & $\approx 27.946$\\
    $\rho^{\frac{\pi}{3}}$  & $45.042$\\
    \hline
  \end{tabular}
\end{table}

Therefore, it is proved that, among triangles with the same diameter, the maximum value of the interpolation constant $C(T)$ in \eqref{def:C_T} happens when the triangle is a regular one.

\section{Conclusions}\label{section:conclusions}

In this paper, the eigenvalue problems of the Laplacian in triangles are considered under two boundary conditions, and it is proved through rigorous computation that the regular triangle minimizes the first eigenvalue of the Laplacian among triangles with a given diameter.  The case of the Dirichlet boundary condition is a reconfirmation of the consequences derived from the results of P\'{o}lya et al. The result of the non-homogeneous Neumann boundary problem shows that the regular triangle gives the maximum value of the interpolation error constant $C(T)$ among triangles of different shapes.

\medskip

\appendix
\section{Perturbation estimates for $k$-th eigenvalue}

Denote by $T$ a triangular domain in $\mathbb{R}^2$, by $S$ an invertible linear transform on $\mathbb{R}^2$.  
Let $(\tilde x,\tilde y)=S(x,y)$ for $(x,y)\in T$, and $\widetilde{T}$ the triangle obtained by applying $S$ to $T$.
For $v$ over $T$, define $\widetilde v$ over $\widetilde{T}$ by
$\widetilde{v}(\widetilde{x}, \widetilde{y}):=v(x,y)$.
Given $V$ as a vector space of function over $T$, let $\widetilde{V}$ be the space obtained by applying $S$ to the function of $V$.

In the argument below, the function space $V$ can be either of $V_0(T)$, $V_e(T)$, $V_{h,0}^{\CG}$, $V_{h,0}^{\CR}$, $V_{h,e}^{\CG}$, and $V_{h,e}^{\CR}$, as defined in the paper.
Note that the boundary condition associated with the space $V$ are well preserved under the transformation $S$. For example, 
for any $v$ in $V_0(T)$, $\widetilde{V}$ has zero value on the boundary, and for any $v$ in $V_e(T)$, $\widetilde{V}$ has zero integral on each edge of $\widetilde{T}$.

Denote $R(T;v)$ and $R(\widetilde{T};\tilde{v})$ by
$$
R(T;v):=\frac{\|\nabla v\|^2_{{T}}}{\|v\|^2_{T}},\quad
R(\widetilde{T};\tilde{v}):=\frac{\|\tilde{\nabla} \tilde{v}\|^2_{\widetilde{T}}}{\|\tilde{v}\|^2_{\widetilde{T}}}~.
$$

\begin{lemma}\label{lem:k-dimensional}
Let $V$ be a $k$-dimensional subspace of $V$. Then, $\mbox{dim}(\widetilde{V})=k$.

\end{lemma}
\begin{proof}
Let $v_1,..,v_k$ be a $L^2$-orthogonal basis of $V$. 
Note that 
\begin{equation*}
    (\widetilde{v}_i, \widetilde{v}_j)_{\widetilde{T}}=(v_i,v_j)_{T}\cdot|\det S|=\delta_{ij}\cdot|\det S|.
\end{equation*}
It is easy to see that $\widetilde{V}_1,..,\widetilde{V}_k$ is a basis of $\widetilde{V}$.
\end{proof}

\begin{lemma}\label{thm:general-perturbation}
Denote by $\lambda_{\min}(SS^\intercal)$ and $\lambda_{\max}(SS^\intercal)$ be the minimum and maximum eigenvalues of $SS^\intercal$, respectively.
Then, we have
\begin{equation*}
\lambda_{\min}(SS^\intercal)\cdot\lambda_k(\widetilde{T})
\leq \lambda_k(T)
\leq \lambda_{\max}(SS^\intercal)\cdot\lambda_k(\widetilde{T}).
\end{equation*}
\end{lemma}

\begin{proof}
Since we have $\nabla v=S^\intercal\widetilde{\nabla}\tilde{v}$, it holds that
\begin{equation*}
\lambda_{\min}(SS^\intercal)\cdot|\widetilde{\nabla}\tilde{v}|^2
\leq |\nabla v|^2 \leq \lambda_{\max}(SS^\intercal)\cdot|\widetilde{\nabla}\tilde{v}|^2.
\end{equation*}
Therefore, we have
\[\displaystyle {\lambda_{\min}(SS^\intercal)}\|\nabla\tilde  v\|^2_{
\widetilde T}\cdot |\det S^{-1}|\leq  \|\nabla v\|^2_{T}\leq {\lambda_{\max}(SS^\intercal)} \|\nabla\widetilde{v}\|^2_{\widetilde T}\cdot |\det S^{-1}|.\] 
Note that $\|v\|^2_{T}=\|\widetilde{v}\|^2_{\widetilde{T}}\cdot |\det S^{-1}|$. Hence, for any $v\in V(T)$,
$$
\lambda_{\min}(SS^\intercal)\cdot R(\widetilde{T};\tilde{v})
\leq R(T;v) \leq \lambda_{\max}(SS^\intercal)\cdot R(\widetilde{T};\tilde{v}).
$$
The mapping $S:V(T)\ni v\mapsto \widetilde{v}\in V(\widetilde{T})$ is injective; see, Theorem 3.41 of \cite{adams2003sobolev}.
By applying Lemma \ref{lem:k-dimensional} and the above inequality, we have
\begin{equation*}
    \lambda_{\min}(SS^\intercal)\cdot \min_{\widetilde{V}^k\subset \widetilde{V}(\widetilde{T})}\max_{\widetilde{v}\in \widetilde{V}^k}R(\widetilde{T};\tilde{v})\leq\min_{V^k\subset V(T)}\max_{v\in V^k}R(T;v),
\end{equation*}
and
\begin{equation*}
    \min_{V^k\subset V(T)}\max_{v\in V^k}R(T;v)\leq\lambda_{\max}(SS^\intercal)\cdot\min_{\widetilde{V}^k\subset \widetilde{V}(\widetilde{T})}\max_{\widetilde{v}\in \widetilde{V}^k}R(\widetilde{T};\tilde{v}).
\end{equation*}
where $\widetilde{V}^k$ and $V^k$ are $k$-dimensional linear subspaces of $\widetilde{V}(\widetilde{T})$ and $V(T)$, respectively.
Thus, by the min-max principle, we have
\begin{equation*}
\lambda_{\min}(SS^\intercal)\cdot\lambda_k(\widetilde{T})
\leq \lambda_k(T)
\leq \lambda_{\max}(SS^\intercal)\cdot\lambda_k(\widetilde{T}).
\end{equation*}
\end{proof}

\medskip

In the following corollary, the statement in Lemma \ref{lem:vertival_monotonicity} is verified.
\begin{corollary}
Given 
$O=(0,0)$, $A=(1,0)$, $B=(x,y_1)$, $\widetilde{B}=(x,y_2)$ such that $y_1\le y_2$.
Let $T=OAB$ and $\widetilde{T}=OA\widetilde{B}$.
Then, it holds that $\lambda_k(T)\geq \lambda_k(\widetilde{T})$.
\end{corollary}

\begin{proof}
It is easy to draw the conclusion by applying Lemma \ref{thm:general-perturbation} along with the transformation  $S=\begin{pmatrix}1 & 0\\ 0 & y_2/y_1\end{pmatrix}$, which satisfies $\lambda_{\min}(SS^\intercal) = 1$. 
\end{proof}

Denote by $T^\theta$ be a triangular domain with vertices as $O=(0,0)$, $A=(1,0)$, $B=(\cos\theta,\sin\theta)$ with $\theta\in (0,\pi)$. In the following corollary, the statement in Remark \ref{rem:h-perturbation} is verified.
\begin{corollary}
For the angles $\theta,\tilde\theta\in (0,\pi)$, let $\widetilde {B}(\cos\tilde\theta,\sin\tilde\theta)$ be a perturbation of $B(\cos\theta,\sin\theta)$. Then, we have
\begin{equation}
\min\left\{\frac{\cos\tilde\theta-1}{\cos\theta-1},\frac{\cos\tilde\theta+1}{\cos\theta+1}\right\}
\cdot\lambda^{\theta}_k
\leq
\lambda^{\tilde{\theta}}_k\leq
\max\left\{\frac{\cos\tilde\theta-1}{\cos\theta-1},\frac{\cos\tilde\theta+1}{\cos\theta+1}\right\}
\cdot\lambda^{\theta}_k.
\end{equation}

\end{corollary}

\begin{proof}
The transform $S$ that maps  $T=OAB$ to $(\tilde x,\tilde y)$ is given by
\begin{equation*}
S=\begin{pmatrix}1 & (\cos\tilde\theta-\cos\theta)/\sin\theta\\
0 &  \sin\tilde\theta/\sin\theta\end{pmatrix}.
\end{equation*}
Note that the eigenvalues of matrix $SS^\intercal$ are given by
$$
\lambda(SS^\intercal) =
\frac{\cos\tilde\theta-1}{\cos\theta-1},~~\frac{\cos\tilde\theta+1}{\cos\theta+1}.
$$
Thus, we draw the conclusion by using Lemma \ref{thm:general-perturbation}.
\end{proof}

\section*{Acknowledgments}
The last author is supported by Japan Society for the Promotion of Science: Fund for the Promotion of Joint International Research (Fostering Joint International Research (A)) 20KK0306, 
Grant-in-Aid for Scientific Research (B) 20H01820, 21H00998, and Grant-in-Aid for Scientific Research (C) 18K03411. 

\bibliographystyle{siamplain}
\bibliography{references}
\end{document}

%% file: shared.tex
\usepackage[normalem]{ulem}
\usepackage{lipsum}
\usepackage{amsmath,amsfonts,amssymb}
\usepackage{graphicx}
\usepackage{epstopdf}
\usepackage{algorithmic}
\ifpdf
  \DeclareGraphicsExtensions{.eps,.pdf,.png,.jpg}
\else
  \DeclareGraphicsExtensions{.eps}
\fi

\newsiamremark{remark}{Remark}
\newsiamremark{hypothesis}{Hypothesis}
\crefname{hypothesis}{Hypothesis}{Hypotheses}
\newsiamthm{claim}{Claim}

\headers{Shape optimization for the Laplacian eigenvalue over triangles}{Ryoki ENDO, Xuefeng LIU}

\title{Shape optimization for the Laplacian eigenvalue over triangles and its application to  interpolation error constant estimation}

\author{Ryoki ENDO\thanks{Graduate School of Science and Technology, Niigata University, Niigata, Japan
  (endo@m.sc.niigata-u.ac.jp).}
\and Xuefeng LIU\thanks{Faculty of Science, Niigata University, Niigata, Japan 
  (xfliu@math.sc.niigata-u.ac.jp).}
}

\usepackage{amsopn}
